\documentclass[11pt]{amsart}

\usepackage{amssymb,epsfig,amsfonts,setspace,dsfont,xcolor, graphicx}
 \usepackage{amsmath,amscd,amsthm,array,geometry}

\usepackage{xr}
\newtheorem{thm}{Theorem}[section]

\newtheorem{prop}[thm]{Proposition}
\newtheorem{conj}[thm]{Conjecture}

\newtheorem{heuristic}[thm]{Heuristic}

\theoremstyle{definition}
\newtheorem{defn}[thm]{Definition}
\newtheorem{rem}[thm]{Remark}

\newcommand{\G}{\mathbf{G}}
\newcommand{\LL}{\mathbf{L}}

\newcommand{\tr}{\mathrm{tr}}
\newcommand{\Hom}{\mathrm{Hom}}
\newcommand{\End}{\mathrm{End}}

\newcommand{\R}{\mathbb{R}}

\newcommand{\GL}{\mathbf{GL}}

\newcommand{\frob}{\mathrm{Frob}}
\newcommand{\Z}{\mathbb{Z}}
\newcommand{\Q}{\mathbb{Q}}
\newcommand{\FF}{\mathbb{F}}

\begin{document}

\title[$\GL_2(\FF_p)$-number fields and elliptic curves of bounded conductor]{On Bhargava's heuristics for $\GL_2(\FF_p)$-number fields and the number of elliptic curves of bounded conductor}
\author{Michael Lipnowski}
\maketitle

\setcounter{section}{0}

\section{Introduction}

The following problems are currently unresolved:

\begin{itemize}
\item[\textbf{Problem ($\Delta$)}]
Prove that the number of elliptic curves over $\mathbb{Q}$ with minimal discriminant of absolute value at most $X$ is $o(X).$ \medskip  

\item[\textbf{Problem ($N$)}]
Prove that the number of elliptic curves over $\mathbb{Q}$ with conductor at most $X$ is at most $o(X).$ \medskip  
\end{itemize}

Let $E_{\mathrm{min}}: y^2 + a_1 xy + a_3 y = x^3 + a_2 x^2 + a_4 x + a_6, a_i \in \mathbb{Z}, a_1,a_3 \in \{ 0,1\}, |a_2| \leq 1$ 
denote the unique integral minimal model for $E.$  Problem ($\Delta$) amounts to counting the number of lattice points $(a_4,a_6)$ in the region \cite{BM} \cite{Watkins}
$$R_X := |64a_4^3 + 432 a_6^2| \leq X.$$  

The area of $R_X$ in the $(a_4,a_6)$-plane has order of magnitude $X^{5/6}$ (with perimeter of smaller order of magnitude), which led Brumer-McGuinness to conjecture that $X^{5/6}$ is the correct order of magnitude in Problem ($\Delta$) \cite{BM}.  
Work of Davenport \cite{Davenport} bounding the number of $\GL_2(\Z)$-orbits of binary cubic forms of $|\mathrm{disc}| < X$ and uniform bounds due to Delone \cite{Delone} and Everste \cite{Evertse} for the number of integral solutions to Thue equations implies that the number of lattice points in $R_X$ is $O(X),$ the best progress toward Problem ($\Delta$) currently known.

Convincing heuristics of Watkins yield an asymptotic guess $\sim c \cdot X^{5/6},$ for an explicit constant $c,$ in Problem ($N$) \cite{Watkins}.  Ultimately, the order of magnitude $X^{5/6}$ again arises from a conjectural asymptotic for the number of lattice points in the region $R_X.$  Duke-Kowalski prove an upper bound of $O_{\epsilon}(X^{1 + \epsilon})$ in Problem ($N$) \cite[\S 3.1]{DK}.

Problems ($\Delta$) and ($N$) have resisted direct geometry-of-numbers approaches because the region $R_X$ is very cuspy, thus making estimating the number of lattice points therein very difficult.  \medskip


One purpose of this note is to explain how Bhargava's local-to-global heuristics for counting number fields shed light on Problem ($N$).  The basic idea can be decomposed into two steps:
\begin{itemize}
\item[Step 1]
Modulo widely believed Diophantine Conjectures\footnote{These conjectures are of a different nature than conjectures pertaining to counting lattice points in the cuspy region $R_X.$} described in \S \ref{diophantineconjectures}, essentially no information is lost in passing from an elliptic curve $E$ to the Galois representation $r_{E,p}: G_{\Q} := \mathrm{Gal}(\overline{\Q} / \Q) \rightarrow \GL_2(\FF_p)$ on its $p$-torsion points (for $p$ sufficiently large, independent of $E$).  

\item[Step 2]
The homomorphisms $r_{E,p}: G_{\Q} \rightarrow \GL_2(\FF_p)$ have severely constrained local properties; these are the local properties enjoyed by every $r_{f,p}$ where $r_{f,p}: G_\Q \rightarrow \GL_2(\FF_p)$ is the Galois representation associated with the $\FF_p$-valued Hecke eigenvalue system $f.$  And indeed, thanks to Khare-Winterberg's proof of Serre's Conjecture, the representations $r_{f,p}$ are the universal source of homomorphisms $r: G_\Q \rightarrow \GL_2(\FF_p)$ which ``look like" the representations $r_{E,p}$ at places of bad reduction and at $p.$  Call representations $r: G_\Q \rightarrow \GL_2(\FF_p)$ satisfying the same local properties as every $r_{f,p}$ \emph{locally modular}. 

This being said, $r_{E,p}$ looks \emph{unlike} a typical $r_{f,p}$ at places of \emph{good} reduction which are small in the archimedean sense.  For example, if $E$ has good reduction at 2, say, the Hasse bound implies that $\tr(r_{E,p}(\frob_2))$ is the mod $p$ reduction of an integer of absolute value at most $2 \sqrt{2}$; this is not a property enjoyed by typical homomorphisms $r: G_\Q \rightarrow \GL_2(\FF_p).$  Bhargava's local-to-global heuristics for counting number fields, recalled in Sections \ref{redux}, \ref{take1} and \ref{take2}, can be used to prove upper bounds on the number of conjugacy classes of homomorphisms $G_{\Q} \rightarrow \GL_2(\FF_p)$ which are locally modular.  Accounting, in addition, for the Hasse bound gives a considerable saving.
\end{itemize}   

Putting together Steps 1 and 2:

\begin{prop} \label{assumingbhargavaheuristic}
Assume Bhargava's local-to-global Heuristics \ref{GL2bhargavaheuristictake1}, \ref{GL2bhargavaheuristictake2} for a subset $U$ of the odd, squarefree integers.  Then 
\begin{equation*}
\limsup_{X \to \infty} \frac{\sum_{N \in U \cap [X,2X]} \# \{ r: G_\Q \rightarrow \GL_2(\FF_p): \mathrm{cond}(r) = N, \text{locally modular}, \text{Hasse bound at 2} \} }{\# U \cap [X, 2X]} \leq \frac{C}{p}
\end{equation*}
for some constant $C.$  If we further assume the Diophantine Conjectures from \S \ref{diophantineconjectures}, then
\begin{equation*}
\limsup_{X \to \infty} \frac{\sum_{N \in U \cap [X,2X]} \# \text{elliptic curves of conductor N} }{\# U \cap [X, 2X]} = 0.  
\end{equation*}
\end{prop}

Short of understanding the convergence rate in Bhargava's heuristics, the above strategy hinges on the fact that the number of locally modular homomorphisms $r: G_\Q \rightarrow \GL_2(\FF_p)$ of fixed squarefree Artin conductor $N$ is bounded on average.  This would follow from the local-to-global Heuristic \ref{GL2bhargavaheuristictake1}.  However, data suggests that there are \emph{too many locally modular homomorphisms} $r: G_\Q \rightarrow \GL_2(\FF_p)$; the number of locally modular $r$ of squarefree conductor $N$ appears to grow quickly with the number of prime factors of $N.$  \medskip

The second purpose of this note is to propose a new model for counting locally modular homomorphisms $r: G_\Q \rightarrow \GL_2(\FF_p)$ of squarefree conductor $N < X.$  Serre's Conjecture tells us that $r$ is associated with an $\FF_p$-valued eigensystem occurring in $S_2(N, \FF_p)_{\mathrm{new}}.$  Suppose $N$ is squarefree with $\omega(N)$ distinct prime factors.  The space $S_2(N, \FF_p)_{\mathrm{new}}$ admits an action by a finite abelian group of Atkin-Lehner involutions, isomorphic to $(\Z / 2)^{\omega(N)},$ which commute with all other (prime to $N$) Hecke operators.  The model proposes that all prime to $N$ Hecke operators, in their action on each of the $2^{\omega(N)}$ simultaneous $\pm 1$ eigenspaces of these Atkin-Lehner involutions, have characteristic polynomials behaving like $2^{\omega(N)}$ independent uniformly sampled (large degree) polynomials over $\FF_p.$  Hecke eigenvalue systems valued in $\FF_p$ correspond to linear factors of these characteristic polynomials.  Since a typical large degree polynomial over $\FF_p$ has one $\FF_p$-rational root, we might guess that there is one locally modular representation per simultaneous Atkin-Lehner eigenspace.  This leads us to the following Conjecture:

\begin{conj} \label{randomlargepolynomial}
Let $U$ be a ``large" subset of the squarefree integers.  Let  
$$G_U(X) :=  \sum_{N \in U \cap [X,2X]} \# \{\text{locally modular } r: G_\Q \rightarrow \GL_2(\FF_p) \text{ of conductor } N \}.$$

There are constants $C_1,C_2 > 0$ for which  

$$C_1 \sum_{N \in U \cap [X,2X]} 2^{\omega(N)} \leq G_U(X) \leq C_2 \sum_{N \in U \cap [X,2X]} 2^{\omega(N)}.$$

In particular, letting $G(X) = G_{\text{all squarefree integers}}(X),$ we conjecture that 

\begin{equation*}
C_1 X \log X \leq G(X) \leq C_2 X \log X.
\end{equation*}
\end{conj}

The effect of the Hasse bound can also be understood from the perspective of this random polynomial model.  For $a_1,\cdots,a_n \in \FF_p$ distinct, those polynomials of degree $d$ vanishing to order at least $k_1,\cdots, k_n$ at $a_1,\ldots, a_n$ respectively lie in the kernels of $k = k_1 + \cdots + k_n$ linear functionals on the space of polynomials of degree $d$ which are linearly independent if $d \geq k.$  For uniformly sampled polynomials of degree $d,$ this simultaneous vanishing occurs with probability $p^{-k}.$  From this, we determine that the expected total order of vanishing at $a_1,\ldots,a_n$ approximately equals $\frac{n}{p} \cdot \frac{p}{p-1}.$  

So for example, if $h_1,\ldots, h_5$ are the mod $p$ reductions of the 5 integers of absolute value at most $2\sqrt{2}$ and the characteristic polynomials of $T_2$ acting on one simultaneous eigenspace $V \subset S_2(N,\FF_p)_{\mathrm{new}}$ of the Atkin-Lehner operators, we would expect the total order of vanishing of this polynomial at $h_1,\ldots, h_5$ to be roughly $\frac{5}{p}.$  Assuming, furthermore, that the characteristic polynomials of $T_2$ acting on all of the $2^{\omega(N)}$ simultaneous eigenspaces behave independently, we are led to the following Conjecture:

\begin{conj} \label{hasserandomlargepolynomial}
Let $U$ be a ``large" subset of the odd squarefree integers.  Let  
$$G'_U(X) :=  \sum_{N \in U \cap [X,2X]} \# \{\text{locally modular } r: G_\Q \rightarrow \GL_2(\FF_p) \text{ of conductor } N,  \tr \; r(\frob_2) \in \{h_1,\cdots, h_5 \} \}.$$

There are constants $C_1,C_2 > 0$ for which  

$$\frac{C_1}{p} \sum_{N \in U \cap [X,2X]} 2^{\omega(N)} \leq G'_U(X) \leq \frac{C_2}{p} \sum_{N \in U \cap [X,2X]} 2^{\omega(N)}.$$

In particular, if $U$ equals the set of all squarefree integers with exactly $s$ prime factors, then
$$\limsup_{X \to \infty} \frac{G'_U(X)}{\# U \cap [X,2X]} \leq \frac{C_2 \cdot 2^s}{p}.$$
\end{conj}

The rough shape of Conjectures \ref{randomlargepolynomial} and \ref{hasserandomlargepolynomial}, when $U$ consists of squarefree integers having a \emph{bounded number of prime factors}, is consistent with Bhargava's local-to-global heuristic for counting locally modular representations $r.$  For such sets $U,$ both models predict a bounded number of locally modular representations per conductor on average.  And for such sets $U,$ both models predict only $\frac{\text{constant}}{p}$ representations per conductor on average which are locally modular and which satisfy the Hasse bound at 2.  However, these two models make genuinely different predictions when $\omega(N)$ is large.

\subsection{Locally modular homomorphisms, integrality of Hecke eigenvalues, and the $O(X^{5/6})$ bound}

Isogeny classes of elliptic curves $E$ of conductor $N$ over $\Q$ are in bijection with Hecke eigenforms in $S_2(N,\overline{\Q})_{\mathrm{new}}$ having \emph{Hecke eigenvalues in $\Z$}  \cite{W} \cite{TW} \cite{BCDT}.
There are roughly $\frac{X^2}{3}$ Hecke eigenforms of level less than $X,$ far more than the roughly $X^{5/6}$ isogeny classes of elliptic curves over $\Q$ of conductor less than $X$ expected in Problem $(N).$  Evidently, it is the integrality of Hecke eigenvalues which severely limits the number of Hecke eigenforms which could possibly be associated with Hecke eigenvalues  

Typically, the minimal field of definition of the Galois representation $r_{f,p}: G_\Q \rightarrow \GL_2(\overline{\FF}_p)$ associated to a Hecke eigenform $f$ has large degree over $\FF_p.$  This is why the upper bound $O(X \log X)$ from Conjecture \ref{randomlargepolynomial} - which an upper bound of $O(X \log X)$ for the number of elliptic curves of conductor les than $X$ modulo the Diophantine conjectures of \S \ref{diophantineconjectures} - is significantly smaller than the $\frac{X^2}{3}$ Hecke eigenforms of level less than $X.$  

On the other hand, the $O(X \log X)$ upper bound for the number of elliptic curves of conductor less than $X$ implied by Conjecture \ref{randomlargepolynomial} is significantly greater than the order of magnitude $X^{5/6}$ predicted for Problem $(N).$  This is because the condition of $r_{f,p}$ having $\FF_p$ as its minimal field of definition does not completely capture the integrality of the Hecke eigenvalues of $f.$  Namely, if $K$ denotes the field generated by the eigenvalues of a Hecke eigenform $f$ and $O_K$ denotes its ring of integers, then every $\Z_p$-factors of $O_K \otimes \Z_p$ gives rise to one locally modular $r$ having minimal field of definition $\FF_p.$     

The improvement of the upper bound from the previous paragraph to $O(\frac{1}{p} X \log X)$ in Conjecture \ref{hasserandomlargepolynomial} reflects a second shadow of the integrality of the eigenvalues of Hecke eigenforms associated with elliptic curves over $\Q.$  Indeed, even if $O_K \otimes \Z_p$ admits a ring homomorphism $\pi: O_K \otimes \Z_p \rightarrow \Z_p,$ there is no reason that $\pi(a_f(2))$ mod $p$ should equal the reduction mod $p$ of a rational integer of archimedean absolute value $\leq 2 \sqrt{2}.$

\subsection{Outline}
\S \ref{diophantineconjectures} reviews several well-known Conjectures which could be used to reduce bounding the number of elliptic curves over $\Q$ with squarefree conductor $< X$ to bounding the number of conjugacy classes of locally modular (surjective) homomorphisms $r: G_\Q \twoheadrightarrow \GL_2(\FF_p)$ of squarefree conductor $< X.$ 

\S \ref{bhargava} reviews Bhargava's local-to-global principle for counting surjective homomorphisms $r: G_\Q \twoheadrightarrow G$ for finite groups $G.$ \S \ref{localpropertiesmodpellipticcurve} describes the local properties of locally modular representations in concrete terms. \S \ref{mass} computes the local masses for locally modular representations.  \S \ref{take1} describes how the local-to-global heuristic predicts a bounded number of locally modular $r$ of fixed conductor on average.  \S \ref{take2} describes the extra saving afforded by imposing the Hasse bound at a fixed small prime.  

\S \ref{atkinlehner} descirbes a different model for counting locally modular representations $r$ accounting for degeneracies in $S_2(N,\FF_p)_{\mathrm{new}}$ caused by Atkin-Lehner involutions. 

\S \ref{data} describes data suggesting that, indeed, the average number of locally modular representations of squarefree Artin conductor $N$ appears to grow like $2^{\omega(N)}.$

\subsection{Acknowledgements}
The author would like to thank Arul Shankar for introducing him to Problems $(N)$ and $(\Delta)$ and for many interesting conversations related to both of these problems.  He would like to thank Lillian Pierce and Melanie Wood for their comments on his letter to them regarding these problems and number field counting heuristics.  He would also like to thank Jordan Ellenberg for enlightening discussions about Malle's and Bhargava's heuristics for counting $\GL_2(\FF_p)$-number fields.    

This paper owes a great mathematical debt to Bhargava's local-to-global heuristics for counting number fields.  The author would like to thank Manjul Bhargava for his inspirational heuristics.

\section{One approach to counting elliptic curves over $\Q$ of squarefree conductor} \label{diophantineconjectures}
Fix a prime $p.$  Let $G_\Q$ be the absolute Galois group of $\Q.$  

\subsection{Counting number fields instead of counting elliptic curves}
Given an elliptic curve $E / \Q,$ consider 
$$r_{E,p}: G_\Q \rightarrow \GL_2(\mathbb{F}_p),$$
the Galois action on $p$-torsion points.  

\begin{thm}[Mazur]
The number of isomorphism classes within every isogeny class of elliptic curves over $\Q$ is uniformly bounded.
\end{thm}

Thus, it suffices to bound the total number of isogeny classes of elliptic curves of conductor less than $X.$

Our main strategy is to bound the total number of isogeny classes of elliptic curves of conductor less than $X$ in terms of the total count of number fields which could conceivably arise as the splitting fields of their $p$-torsion Galois representations.  The Frey-Mazur conjecture implies that essentially no information is lost in the passage to $p$-torsion.   
\begin{conj}[Frey-Mazur]
Suppose $p \geq 17.$  The assignment
$$\text{ isogeny class of } E \leadsto r_{E,p}$$
is injective.
\end{conj}

\begin{rem}
Remarkable recent progress has been made on the Frey-Mazur conjecture for elliptic curves over the function field of a complex curve by Bakker-Tsimerman \cite{BT}.  For our present purposes, we will only need the Frey-Mazur theorem ``on average":
\begin{align*}
&\# \{ \text{isogeny classes of elliptic curves over } \Q \text{ of conductor } < X  \} \\
&\leq \# \{ \GL_2(\FF_p) \text{-number fields for which the associated } r: G_\Q \rightarrow \GL_2(\FF_p) \\
& \text{of conductor } < X \text{ satisfies ``the right local properties"}  \}
\end{align*}
\end{rem}

Bhargava's heuristics pertain to \emph{surjective} homomorphisms $r: G_\Q \twoheadrightarrow \GL_2(\FF_p).$  It is widely believed that $r_{E,p}$ is surjective for $p$ sufficiently large: 
\begin{conj}[Serre's uniformity conjecture]
Suppose $p > 37.$  Then the homomorphism
$$r_{E,p}: G_\Q \twoheadrightarrow \GL_2(\FF_p)$$
is surjective.
\end{conj}

\section{Estimating the number of $\GL_2(\FF_p)$-number fields with good local properties via Bhargava's heuristics} \label{bhargava}
\subsection{Redux of Bhargava's heuristics} \label{redux}
Let $G_\Q := \mathrm{Gal}(\overline{\Q}/ \Q).$  For a place $v,$ let $G_v$ be the decomposition group at $v.$

Let $G$ be a finite group.  Let $\LL_v  := \Hom(G_v, G)^\#$ denote the finite set of conjugacy classes of homomorphisms $G_v \rightarrow G.$  There is a natural measure on $\LL_v$ given by $m_v(\{ \phi \}) = \frac{1}{|Z_G(\phi)|}.$  The mass of the unramified homomorphisms in $\LL_v$ equals 1.  We form  $\LL = \prod'_v \LL_v,$ where the product is restricted with respect to the unramified homomorphisms.  We endow $\LL$ with the product measure $m = \prod'_v m_v.$

Let $\G$ denote the conjugacy classes of surjective homomorphisms $G_{\Q} \twoheadrightarrow G.$  There is a natural localization map
$$\mathrm{loc}: \G \rightarrow \LL.$$  

\begin{heuristic}[Bhargava's local-global heuristic for counting $G$-number fields] \label{generalbhargavaheuristic}
For large subsets $U \subset \LL,$
\begin{equation} \label{bhargavaheuristic}
\sum_{r: \mathrm{loc}(r) \in U} \frac{1}{|Z_G(r)|} \sim m(U).
\end{equation}
\end{heuristic}

Bhargava \cite{Bhargava} originally applied this heuristic to conjecture asymptotic formulas for  $S_n$-number fields of bounded discriminant.  See \cite[\S 6]{BV} for some discussion on applying this heuristic in other contexts, especially for $G$ a finite group of Lie type.  \medskip

We will apply Heuristic \ref{generalbhargavaheuristic} for $G = \GL_2(\FF_p)$ in \S \ref{take1}, \ref{take2} to give a heuristic count for the number of $\GL_2(\FF_p)$-number fields with local properties mimicking those the $r_{E,p}.$

\subsection{Local properties of $r_{E,p}$} \label{localpropertiesmodpellipticcurve}
The representation $r_{E,p}$ has well-understood local properties.  Fix a finite set of primes $S$ not including $p.$  Suppose $E$ has good reduction outside of $S$ and has squarefree conductor $N(S) = \prod_{v \in S} v$ for some finite set of primes $S.$  Then
\begin{itemize}
\item[(a)]
Suppose $v \notin S \cup \{ p \}.$  Then $r_{E,p}$ is unramified at $v.$

\item[(b)]
Suppose $v \in S.$  Then $r_{E,p}$ is \emph{Steinberg} at $v,$ i.e. its restriction to inertia at $v$ generates a principal unipotent subgroup of $\GL_2(\FF_p).$  This implies, in particular, that $r_{E,p}$ is tamely ramified. 

\item[(c)]
Suppose $v = \infty.$  Then complex conjugation is \emph{odd},  
$$r_{E,p}(\text{complex conjugation}) \sim \left( \begin{array}{cc} 
1 & 0 \\
0 & -1
\end{array} \right)$$

\item[(d)]
Suppose $v = p.$  Because $E$ has good reduction at $p,$ its reduction is either ordinary or supersingular.  If its reduction is \emph{ordinary}, then 
$$r_{E,p}|_{G_p} \sim \left( \begin{array}{cc} \chi \cdot e & \ast \\ 0 & e^{-1}  \end{array} \right),$$
where $\chi$ denotes the mod $p$ cyclotomic character and $e$ is an unramified character of $G_p.$  

If $E$ has good \emph{supersingular} reduction, then $r_{E,p}|_{G_p}$ is irreducible, tamely ramified, and over $\overline{\FF}_p$ becomes the direct sum of the two (Galois conjugate) fundamental characters of tame inertia of level $2.$ In particular, there is only one such representation and its centralizer is the center of $\GL_2(\mathbb{F}_p)$ \cite{Serre}.
\end{itemize}

\subsection{Mass computations} \label{mass} 
We next compute the (local) masses of the conjugacy classes of homomorphisms $r: G_v \rightarrow \GL_2(\mathbb{F}_p)$ corresponding to the local conditions from (a), (b), (c), and (d).  Similar computations are done in \cite[\S 6]{BV}:

\begin{itemize}
\item[(a$'$)] 
Suppose $v \notin S \cup \{ p \}$:
The mass of all unramified representations $r: G_v \rightarrow \GL_2(\mathbb{F}_p)$ equals $1.$

\item[(b$'$)]
For $v \in S$:  The mass of all Steinberg homomorphisms
$r: G_v \rightarrow \GL_2(\FF_p)$
is 1.   

Indeed, let $F, U$ be generators of tame inertia at $v$ with $F U F^{-1} = U^v.$  Since $r$ is Steinberg, we may assume with no loss of generality that 
$$r(U) = \left( \begin{array}{cc} 1 & 1 \\ 0 & 1 \end{array} \right).$$

The above relation between $F,U$ implies that 
$$r(F) = \left( \begin{array}{cc} v \cdot d &  b \\ 0 & d \end{array} \right)$$
for some $d \in \FF_p^{\times}.$  
\begin{itemize}
\item
Suppose $v = 1 \mod p.$  Then $r(F)$ lies in the centralizer of $r(U),$ which is abelian.  Thus, as $b,d$ vary, all of the above representations $r$ are non-conjugate.  Furthermore, every $r$ above has centralizer of order $p(p-1).$  Therefore, the mass at $v$ of Steinberg representations equals $p(p-1) / p(p-1) = 1$ in this case.

\item
Suppose $v \neq 1 \mod p.$  The matrix $\left(\begin{array}{cc} 1 & \frac{b}{d (v-1)} \\ 0 & 1 \end{array} \right)$ conjugates $r$ to the representation
$$r'(U) =\left( \begin{array}{cc} 1 & 1 \\ 0 & 1 \end{array} \right) , r'(F) = \left( \begin{array}{cc} v\cdot d & 0 \\ 0 & d \end{array} \right).$$
The $p-1$ representations $r'$ are all non-conjugate and the centralizer of $r'$ equals the center of $\GL_2(\FF_p).$  Therefore,
the mass at $v$ of Steinberg representations equals $(p-1) / (p-1) = 1$ in this case too.
\end{itemize}

\item[(c$'$)]
Suppose $v = \infty.$  The centralizer in $\GL_2(\mathbb{F}_p)$ of $\left( \begin{array}{cc} 1 & 0 \\ 0 & -1 \end{array} \right)$ has order $(p-1)^2.$  Therefore,  the mass of all odd representations of $G_{\infty}$ equals $\frac{1}{(p-1)^2}.$

\item[(d$'$)]
Suppose $v = p,$ the most interesting case.  We first compute the mass of ordinary representations, i.e. those conjugate to 
$$r \sim \left( \begin{array}{cc} \chi \cdot e  & \ast \\ 0 & e^{-1} \end{array} \right).$$
where $\chi$ is the mod p cyclotomic character and $e$ is an unramified character $G_p \rightarrow \FF_p^{\times}.$  Conjugacy classes of such homomorphisms are in bijection with extensions

$$\mathrm{Ext}^1(e^{-1} , \chi \cdot e) = \mathrm{Ext}^1(1, \chi \cdot e^2) = H^1( G_p, \chi \cdot e^2).$$

A Galois cohomology computation shows that the mass of all ordinary representations of $G_p$ has size $3p-1 + \frac{1}{p-1}.$  See \S \ref{massp}. 

As commented in (d), the mass of supersingular representations $r: G_p \rightarrow \GL_2(\mathbb{F}_p)$ equals $\frac{1}{p-1}.$   
\end{itemize}

\begin{defn}
Call representations $r: G_\Q \rightarrow \GL_2(\FF_p)$ satisfying the local conditions (a$'$),(b$'$),(c$'$),(d$'$) \emph{locally modular with respect to $S$} (or just \emph{locally modular} if $S$ is understood).
\end{defn}

\subsection{Bhargava's heuristics, take 1} \label{take1}
Let $G(S)$ denote the number of surjective homomorphisms $r: G_\Q \twoheadrightarrow \GL_2(\mathbb{F}_p)$ which are locally modular with respect to $S.$  According to Bhargava's Heuristic \ref{generalbhargavaheuristic}, the mass of such global Galois representations, which equals  
$$\sum_{r \text{ locally modular w.r.t. } S} \frac{1}{ | \text{ centralizer in } \GL_2(\mathbb{F}_p) \text{ of } r | } = \frac{1}{p-1} G(S),$$
should asymptotically equal $m(S)$ on average, i.e.

\begin{heuristic}[local-to-global heuristic for counting $\GL_2(\FF_p)$-number fields] \label{GL2bhargavaheuristictake1}
For ``large" subsets $U$ of the (squarefree) integers 

\begin{equation} \label{bhargavaheuristictake1}
\frac{1}{p-1} \sum_{N(S) \in U \cap [1,X]} G(S) \sim \sum_{N(S) \in U \cap [1,X]} m(S) \text{ where } N(S) := \prod_{p \in S} p.
\end{equation}
\end{heuristic}

Combining our calculations from the previous section,  

\begin{align} \label{finallocalmass}
m(S) &= \mathrm{mass}_{\infty} \cdot \mathrm{mass}_p \cdot \prod_{v \in S} \mathrm{mass}_v  \cdot \prod_{v \notin S} \mathrm{mass}_v \nonumber \\
&= \frac{1}{(p-1)^2} \cdot \left( 3p - 1 + \frac{2}{p-1} \right) \cdot 1 \cdot 1.
\end{align}
According to Heuristic \ref{GL2bhargavaheuristictake1} and \eqref{finallocalmass}, the number of Galois representations locally modular with respect to $S$ should be at most about $3$ on average.\footnote{Our mass computation assumes that $r|_{G_p}$ was either good ordinary or singular.  It didn't account for elliptic curves of bad reduction.  But since we're thinking of $p$ as a fixed large prime, almost all Galois representations associated with elliptic curves will fall into either the good ordinary or good supersingular masses accounted for above.}

\emph{Assuming Heuristic \ref{GL2bhargavaheuristictake1} for the set $U$ of all squarefree integers} would imply Thue's average $O(1)$ result; it would improve upon the average $O_{\epsilon}(X^{\epsilon})$ from \cite{DK}; but it would fail to capture the average $o(1)$ result required for Problems $(N)$ and $(\Delta).$  \medskip

We believe that when $U = $ all squarefree integers, Heuristic \ref{GL2bhargavaheuristictake1} may undercount the number of $\GL_2(\FF_p)$-extensions satisfying the expected local properties.  This will be discussed further in \S \ref{atkinlehner}.

\subsection{Bhargava's heuristics, take 2} \label{take2}
The elliptic curves of conductor less than $X$ are \emph{far more locally constrained at unramified small primes} than we've accounted for.  Namely, the Frobenius trace $a_v(E)$ satisfies the Hasse bound: it is an integer of absolute value at most $ 2\sqrt{ v }.$  So if $v$ is less than $\frac{p^2}{16},$ then $\tr \; r_{E,p}(\mathrm{Frob}_v) = a_v(E) \mod p$ is non-trivially constrained.

Accounting for this Hasse bound constraint, Bhargava's heuristics suggest 

\begin{heuristic}[local-to-global heuristic for counting $\GL_2(\FF_p)$-number fields satisfying the Hasse bound at 5] \label{GL2bhargavaheuristictake2} 
Let $G'(S)$ denote the number of conjugacy classes of homomorphisms $r: G_\Q \rightarrow \GL_2(\FF_p)$ which are locally modular with respect to $S,$ which have odd conductor $N(S),$ and which satisfy the Hasse bound at 2, i.e. $\tr( r(\frob_2))$ is the mod $p$ reduction of an integer $n$ satisfying $|n| \leq \lfloor 2 \sqrt{2} \rfloor = 2.$ For ``large" subsets $U$ of the odd squarefree integers, 
$$\sum_{N(S) \in U \cap [1,X]} \frac{1}{p-1} G'(S) \sim \sum_{N(S) \in U \cap [1,X]} m(S) \cdot \text{probability a conjugacy class satisfies the Hasse bound at } 2.$$
\end{heuristic}

\begin{prop}
Let $U$ be a subset of the squarefree integers.  Let $U^{(2,p)}$ denote $U$ with all multiples of $2$ and $p$ removed.  Assume Heuristic \ref{GL2bhargavaheuristictake2}.  There is a constant $C$ for which 
$$\limsup_{X \to \infty} \frac{ \# \{ \text{elliptic curves with squarefree conductor } \in U^{(2,p)} \cap [1, X] \} }{\# U^{(2,p)} \cap [1,X]} \leq \frac{C}{p}.$$
In particular, if $U^{(2,p)}$ has density in $U$ at least $\epsilon > 0$ for all $p$ and if the set of elliptic curves of conductor in $U^{(2,p)}$ has density at least $\epsilon' > 0$ in the set of elliptic curves of conductor in $U,$ then 
$$\limsup_{X \to \infty} \frac{ \# \{ \text{elliptic curves with squarefree conductor } \in U \cap [1, X] \} }{\# U \cap [1,X]} = 0.$$
\end{prop}

\begin{proof}
By direct calculation, the number of elements of trace 0 in $\GL_2(\FF_p)$ equals $p^2(p-1).$  The number of elements of trace $a$ for every $a \in \FF_p^\times$ equals 
$$\frac{\# \GL_2(\FF_p) - \# \text{trace 0 elments}}{p-1} = p(p^2 - p - 1).$$
Thus, the likelihood a random conjugacy class satisfies the Hasse bound at 2 equals
$$\frac{4 \cdot  p(p^2 - p - 1) + 1 \cdot p^2(p-1)}{\# \GL_2(\FF_p)} = \frac{5}{p} + O(p^{-2}).$$

Applying Heuristic \ref{GL2bhargavaheuristictake2},
\begin{align} \label{bhargavaheuristictake2}
&\# \{ \text{elliptic curves with squarefree conductor } \in U^{(2,p)} \cap [1,X]  \} \nonumber \\
&\leq \sum_{S: N(S) \in U^{(2,p)} \cap [1,X] } G'(S) \nonumber \\
&\sim \sum_{S: N(S) \in U^{(2,p)} \cap [1,X]} (p-1) \cdot m(S) \cdot \text{likelihood of satisfying the Hasse bound at } 2 \nonumber \\
&= \frac{1}{p-1} \cdot \left( 3p-1 + \frac{2}{p-1} \right)  \cdot \left( \frac{5}{p} + O(p^{-2}) \right)  \cdot \# U^{(2,p)} \cap [1,X]. 
\end{align}

It follows that for every fixed $p,$  
$$\limsup_{X \to \infty} \frac{\# \{ \text{elliptic curves with squarefree conductor } \in U^{(2,p)} \cap [1,X] \} }{\# U^{(2,p)} \cap [1,X]} \leq \frac{C}{p}$$
for some absolute constant $C.$  If we assume that $U^{(2,p)}$ has positive density in $U$ and that the set of elliptic curves of conductor in $U^{(2,p)}$ has positive density in the set of elliptic curves of density $p,$ then letting $p \to \infty$ gives  
$$\lim_{X \to \infty} \frac{ \# \{ \text{elliptic curves with squarefree conductor } \in U \cap [1,X] \}}{\# U \cap [1,X]} = 0.$$ 
\end{proof} 

\emph{Assuming Heuristic \ref{GL2bhargavaheuristictake2} for $U = $ all squarefree integers}, the above calculation resolves Problems $(N)$ and $(\Delta).$   However, we believe that when $U = $ all squarefree integers, Heuristic \ref{GL2bhargavaheuristictake2} may undercount the number of $\GL_2(\FF_p)$-extensions satisfying the expected local properties.  This will be discussed further in \S \ref{atkinlehner}.

\section{The effect of Atkin-Lehner involutions on the number of $\GL_2(\FF_p)$-extensions with good local properties} \label{atkinlehner}

We propose a second model for counting conjugacy classes of surjective homomorphisms $r: G_\Q \twoheadrightarrow \GL_2(\FF_p)$ which are locally modular.    

Let $N$ denote the Artin conductor of $r$; suppose that $p \nmid N.$  Let $S_2(N, \FF_p)_{\mathrm{new}}$ denote the space of weight 2 cusp forms with $\FF_p$-coefficients. Let $\mathbb{T}$ denote the $\FF_p$-subalgebra of $\End(S_2(N, \FF_p)_{\mathrm{new}})$ generated by the prime to $N$ Hecke operators and let $\mathbb{T}_{\mathrm{red}}$ denote its reduced quotient.  By Serre's Conjecture - now proven by Khare-Winterberger \cite{KW} - there is a bijection 

\begin{align} \label{serre}
&\text{Locally modular homomorphisms } r: G_\Q \rightarrow \GL_2(\FF_p) \nonumber\\
\leftrightarrow &\FF_p\text{-algebra homomorphisms } \phi: \mathbb{T} \rightarrow \FF_p. \nonumber\\
= &\FF_p\text{-algebra homomorphisms } \phi: \mathbb{T}_{\mathrm{red}} \rightarrow \FF_p
\end{align} 

characterized by the compatibility that if $r$ is associated with $\phi,$ then for all primes $v \nmid Np,$ the characteristic polynomial of $r(\frob_v) = x^2 - \phi(T_v) x + v.$  

\begin{rem}
The conditions defining ``locally modular" were concocted specifically to match the local properties of the representations $r_{\phi}: G_\Q \rightarrow \GL_2(\FF_p)$ associated with homomorphisms $\phi: \mathbb{T} \rightarrow \FF_p.$
\end{rem}

\subsection{Atkin-Lehner involutions}
Suppose $N = p_1 \cdots p_s$ for distinct primes $p_1,\ldots, p_s.$  For every $i = 1,\cdots, s,$ there is an associated Atkin-Lehner involution $w_i,$ which we briefly describe.  

The curve $X_0(N)$ parametrizes pairs $(E,C)$ where $E$ is an (generalized) elliptic curve and $C \subset E[N]$ is cyclic of order $N.$  There is a canonical decomposition $C = C_1 \times \cdots \times C_s,$ where $C_i \subset E[p_i]$ is cyclic of order $p_i.$  Then

$$w_i(E,C) := (E, C_1 \times \cdots \times C_i^\perp \times \cdots \times C_s),$$

where $C_i^\perp$ is the orthogonal complement of $C_i$ in $E[p_i]$ with respect to the $p_i$-Weil pairing.  All $w_i$ commute with each other, commute with $\mathbb{T},$ and satisfy $w_i^2 = 1.$  Let $W$ denote the finite group of automorphisms of $S_2(N,\FF_p)_{\mathrm{new}},$ isomorphic to $\left(\mathbb{Z}/2 \right)^s,$ generated by the Atkin-Lehner operators.  Let $\widehat{W}$ denote its character group.  There is an isotypic decomposition  

$$S_2(N,\FF_p)_{\mathrm{new}} = \bigoplus_{\chi \in \widehat{W}} S_2(N,\FF_p)_{\mathrm{new},\chi}$$ 

and corresponding decompositions

$$\mathbb{T} = \prod_{\chi \in \widehat{W}} \mathbb{T}_{\chi} \text{ and } \mathbb{T}_{\mathrm{red}} = \prod_{\chi \in \widehat{W}} \mathbb{T}_{\mathrm{red}, \chi}.$$

Let $d = \dim_{\FF_p} S_2(N,\FF_p)_{\mathrm{new}} = \dim_{\FF_p} \mathbb{T},$ let $d_{\chi} = \dim_{\FF_p} S_2(N,\FF_p)_{\mathrm{new},\chi} = \dim_{\FF_p} \mathbb{T}_{\chi},$ and let $e_{\chi} = \dim_{\FF_p} \mathbb{T}_{\mathrm{red}, \chi}.$  We \emph{assume the dimensions $e_{\chi}$ are all large, say $\gg N^{\theta}$ for some $\theta > 0.$}

\begin{rem} 
Applying the trace formula to compute the trace of products $w_{i_1} \cdots w_{i_k}$ acting on $S_2(N, \Q)_{\mathrm{new}},$ Martin \cite{Martin} shows that $d_{\chi, \Q} := \dim_\Q S_2(N,\Q)_{\mathrm{new},\chi}$ all roughly equal $d / 2^s.$  This does not directly imply the above assumption, since $d_{\chi,\Q}$ does not generally equal $d_{\chi}.$  For one, the mod $p$ reduction of the new subspace of $S_2(N,\Q)$ might not equal the new subspace of $S_2(N,\FF_p)$ because of mod $p$ congruences.  Also, it is not always true that $e_\chi = d_\chi.$  However, Martin's result strongly suggests that the above largeness hypothesis is reasonable. 
\end{rem}

\subsection{Modelling the $\mathbb{T}_{\mathrm{red},\chi}$ using random permutations} \label{heckealgebrarandompermutation}
The $\FF_p$-algebra $\mathbb{T}_{\mathrm{red}, \chi}$ is finite \'{e}tale.  Isomorphism classes of finite \'{e}tale $\mathbb{F}_p$-algebras of degree $e$ are equivalent to isomorphism classes of $\Gamma = \mathrm{Gal}(\overline{\FF}_p / \FF_p)$-sets of size $e$ (with open stabilizers).  Representing isomorphism classes of $\Gamma$-sets of size $e$ as conjugacy classes of (continuous) homomorphisms $\Gamma \rightarrow S_e,$ it follows that isomorphism classes of $\Gamma$-sets are uniquely determined by the conjugacy class of $\frob_p$ in $S_e$; for a finite \'{e}tale algebra $E$ of degree $e,$ we let $\frob_p(E)$ denote its associated conjugacy class in $S_e.$ 

\begin{heuristic}[Random independent permutation model] \label{randompermutationmodel}
We model the collection of isomorphism classes $\mathbb{T}_{\mathrm{red},\chi}, \chi \in \widehat{W},$ as the finite \'{e}tale $\FF_p$-algebra associated with independently sampled random conjugacy classes $c_\chi \subset S_{e_\chi}, \chi \in \widehat{W}.$

``Random conjugacy class" means that we assign $c_{\chi}$ mass $\frac{\# c_{\chi}}{e_\chi !}$; this is the pushforward to conjugacy classes of the uniform probability measure on $S_{e_\chi}.$
\end{heuristic}

For every finite \'{e}tale $\FF_p$ algebra $E$ of degee $e,$ there is a bijection
\begin{equation*}
\FF_p\text{-algebra homomorphisms } (E \rightarrow \FF_p) \leftrightarrow \text{fixed points of } \frob_p(E). 
\end{equation*}
It is well-known that the number of fixed points $X$ of a uniformly sampled element of $S_n$ converges in distribution to $\mathrm{Poisson}(1),$ the Poisson distribution of mean 1 for which $\mathrm{Prob}(X = k) = e^{-1} \frac{1}{k!},$ rapidly as $n$ grows.  In distribution, our random model thus predicts

\begin{align} \label{permutationmodel}
&\# \{ r: G_\Q \rightarrow \GL_2(\FF_p) \text{ locally modular of conductor } N = p_1 \cdots p_s \} \nonumber \\
&= \sum_{\chi \in \widehat{W}} \# \{\text{Homomorphisms } \phi: \mathbb{T}_{\mathrm{red},\chi} \rightarrow \FF_p  \} \nonumber \\
&= \sum_{\chi \in \widehat{W}} \# \text{fixed points of  } \frob_p(\mathbb{T}_{\mathrm{red},\chi}) \nonumber \\
&\sim \sum_{\chi \in \widehat{W}} \text{independent } \mathrm{Poisson}(1) \hspace{0.5cm} \text{ by Heuristic \ref{randompermutationmodel} } \nonumber\\
&\sim \mathrm{Poisson}(2^s).
\end{align}

The approximation from the second last line above, that that number of fixed points of a random permutation nearly equals $\mathrm{Poisson}(1)$ in distribution, is very accurate if $e_{\chi}$ is large.  

Summing over all squarefree $N$ in the interval $[X,2X],$ \eqref{permutationmodel} suggests that

\begin{align} \label{resultofpermutationmodel}
&\sum_{N \text{ squarefree } \in [X,2X]} \# \{ r: G_\Q \rightarrow \GL_2(\FF_p) \text{ locally modular of conductor } N = p_1 \cdots p_s \} \nonumber \\
&\sim \sum_{N \text{ squarefree } \in [X,2X]} 2^{\omega(N)},
\end{align} 

where $\omega(N)$ denotes the number of distinct prime factors of $N.$\footnote{If we further assume, that the counts in \eqref{permutationmodel} are independent, then grouping the sum into terms with $\omega(N)$ constant, the central limit theorem implies that the passage from \eqref{permutationmodel}  to \eqref{resultofpermutationmodel} should have only square root error with very high probability.}  This leads us to make the following Conjecture based on \eqref{resultofpermutationmodel}:  

\begin{conj} \label{resultofnewmodelGL2fields}
Let 
$$G(X) :=  \sum_{N \text{ squarefree } \in [X,2X]} \# \{ r: G_\Q \rightarrow \GL_2(\FF_p) \text{ locally modular of conductor } N = p_1 \cdots p_s \}.$$

There are constants $C_1,C_2 > 0$ for which  

\begin{equation*}
C_1 X \log X \leq G(X) \leq C_2 X \log X.
\end{equation*}
\end{conj}

Conjecture \ref{resultofnewmodelGL2fields} would follow from \eqref{resultofpermutationmodel} because $\sum_{N \text{ squarefree } \in [X,2X]} 2^{\omega(N)}$ has order of magnitude $X \log X$; see \S \ref{tauberiantheorem} for details.

We emphasize that \textbf{Conjecture \ref{resultofnewmodelGL2fields} lies at ends with the local-to-global Heuristic \ref{GL2bhargavaheuristictake1} as applied to the set $U$ of all squarefree integers}.  The next section describes some data pertaining to $\GL_2(\FF_p)$-number fields which suggests that the permutation model posited in this section may hold water.

\section{Some data for $\GL_2(\FF_{101})$-number fields} \label{data}
\subsection{Setup} \label{setup}
For an integer $N,$ let $m_{\mathrm{new}}(N)$ denote the number of \emph{multiplicity one} linear factors of the characteristic polynomial of $T_2$ acting on $S_2(N; \FF_{101})$ which do not divide the characteristic polynomial of $T_2$ acting on $S_2(M, \FF_{101})$ for any proper divisor $M | N.$  


For all odd squarefree integers $2500 \leq N \leq 5000$ which are products of at most three prime factors, we calculated $m_{\mathrm{new}}(N)$ in SAGE.  There are $302$ such $N$ which are prime, $489$ such $N$ with exactly two prime factors, and $203$ such $N$ with exactly three prime factors. \medskip

The integer $m_{\mathrm{new}}(N)$ is always \emph{lower bound} for the number $r(N)$ of distinct conjugacy classes of locally modular homomorphisms $r: G_\Q \rightarrow \GL_2(\FF_{101})$ of conductor $N.$  It could be smaller for a couple reasons:
\begin{itemize}
\item
There could be two distinct Galois representations $r,r': G_\Q \rightarrow \GL_2(\FF_{101})$ of conductor $N$ for which $\tr(r(\frob_2)) = \tr(r'(\frob_2)).$  Then $r,r'$ together would contribute 0 to $m_{\mathrm{new}}(N)$ but contribute 2 to $r(N).$  This possibility is quite unlikely because 101 is large.

\item
A representation $r: G_\Q \rightarrow \GL_2(\FF_{101})$ of conductor $N$ could satisfy $\tr(r(\frob_2)) = \tr(r'(\frob_2))$ for some representation $r': G_\Q \rightarrow \GL_2(\FF_{101})$ of conductor $N'$ properly dividing $N.$  Then $r$ would contribute 0 to $m_{\mathrm{new}}(N)$ but contribute 1 to $r(N).$  This possibility is again quite unlikely because 101 is quite large and because $S_2(N',\FF_{101})_{\mathrm{new}}$ has (sometimes much) smaller dimension than $S_2(N,\FF_{101})_{\mathrm{new}}.$
\end{itemize}

To rigorously compute $r(N)$ via modular forms requires working with Hecke operators themselves (i.e. understanding their full action on modular symbols, or $q$-expansions, or \ldots) and finding a simultaneous Jordan decomposition.  This is far more computationally intensive than computing the characteristic polynomial of a single low degree Hecke operator.  However, $m_{\mathrm{new}}(N)$ appears to be a reasonable proxy for $r(N).$   

\begin{figure}
  \includegraphics[width=\linewidth]{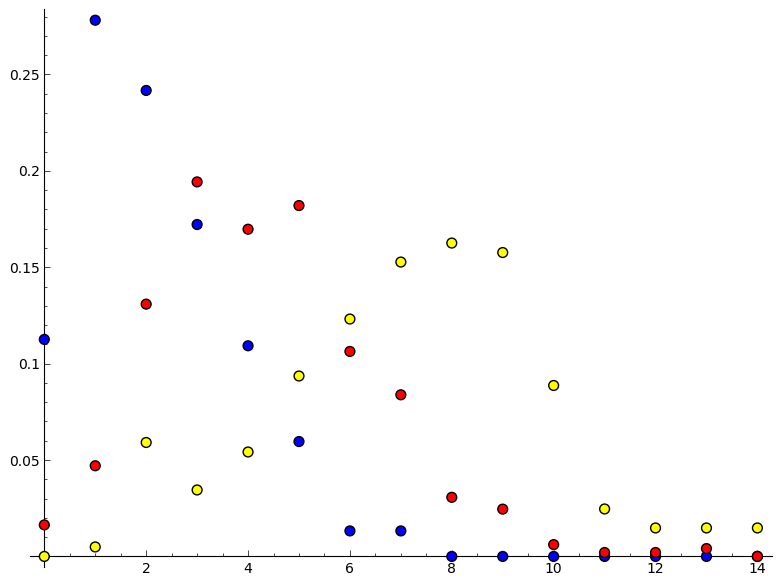}
  \caption{The distribution of $m_{\mathrm{new}}(N), 2500 \leq N \leq 5000,$ for $\omega(N) = s$ for $s = 1,2,3$ plotted above in blue, red, and yellow respectively.}
  \label{simplelinearfactorslevelatmostthreeprimefactors}
\end{figure}

\subsection{Explanation of Figure \ref{simplelinearfactorslevelatmostthreeprimefactors}} 
For fixed $s,$ the point $(x,y)$ appears in Figure \ref{simplelinearfactorslevelatmostthreeprimefactors} if 


\begin{equation} \label{meaningofdotssimplelinearfactorslevelatmostthreeprimefactors}
y = \frac{\# \{ 2500 \leq N \leq 5000: N = p_1 \cdots p_s \text{ for } p_1,\ldots,p_s \text{ distinct odd primes}, m_{\mathrm{new}}(N) = x \} }{\# \{ 2500 \leq N \leq 5000 : N = p_1 \cdots p_s \text{ for } p_1, \ldots, p_s  \text{ distinct odd primes} \} }.
\end{equation}

Blue dots correspond to $s = 1,$ red dots correspond to $s = 2,$ and yellow dots correspond to $s = 3.$  Per our summary in \S \ref{setup}, the denominator in \eqref{meaningofdotssimplelinearfactorslevelatmostthreeprimefactors} 
equals 302, 489, and 203 respectively for $s = 1,2,3.$

Qualitatively, the data appears to be centered from left to right in the order blue, red, yellow.  This is quantified by the following summary statistics: \bigskip


\begin{tabular}{l | c | c | c} 
$s$   & color & mean of $m_{\mathrm{new}}(N)$ for $\omega(N) = s$ & variance of $m_{\mathrm{new}}$ for $\omega(N) = s$ \\
\hline
1 & blue & 2.18543 \ldots & 2.382834\ldots \\
\hline
2 & red & 4.33333 \ldots & 4.54124\ldots \\
\hline
3 & yellow & 7.1724\ldots & 6.773229\ldots \\
\end{tabular}
\bigskip

The above means for $m_{\mathrm{new}}(N),$ for $\omega(N) = s$ fixed, are ``in the same ballpark" as the respective means of $2^1, 2^2,$ and $2^3$ that we'd expect for $s = 1,2,$ and $3$ from the random permutation Heuristic \ref{randompermutationmodel}.

By contrast, if the statistics of locally modular representations $r: G_\Q \rightarrow \GL_2(\FF_{101})$ were governed by Heuristic \ref{GL2bhargavaheuristictake1}, for $U = $ all squarefree integers in the notation of \S \ref{take1}, we would expect the mean to be bounded \emph{independent of} $s.$

\subsection{Numerical evidence that the number of $\FF_p$-valued Hecke eigenvalue systems occuring in $S_2(N,\FF_{101})_{\mathrm{new},\chi}$ for $\chi \in \widehat{W}$ has distribution $\mathrm{Poisson}(1)$}

Suppose $N = p_1 \cdots p_s$ for distinct primes $p_1,\cdots,p_s.$  Heuristic \ref{randompermutationmodel} predicts that the number of $\FF_p$-valued Hecke eigenvalue systems occuring in every simultaneous Atkin-Lehner eigenspace $S_2(N, \FF)_{\mathrm{new},\chi}, \chi \in \widehat{W},$ should have distribution $\mathrm{Poisson}(1)$ and furthermore that these counts should behave independently; see \S \ref{heckealgebrarandompermutation} for further details.  Since the sum of $2^s$ independent $\mathrm{Poisson}(1)$ random variables has distribution $\mathrm{Poisson}(2^s),$  Heuristic \ref{randompermutationmodel} suggests that $m_{\mathrm{new}}(N)$ should have distribution $\mathrm{Poisson}(2^s).$

\begin{figure}
  \includegraphics[width=\linewidth]{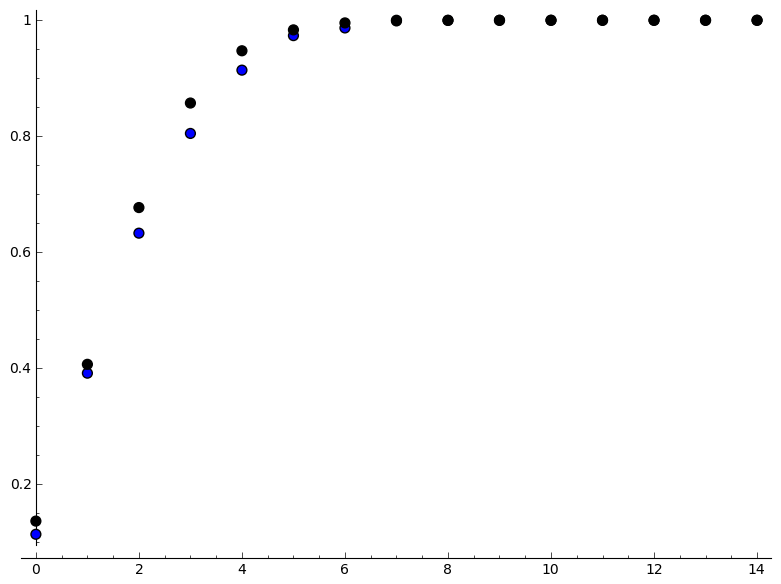}
  \caption{Comparing the CDF of the blue dot distribution of $m_{\mathrm{new}}(N), 2500 \leq N \leq 5000,$ for $N$ prime from Figure \ref{simplelinearfactorslevelatmostthreeprimefactors}, pictured in blue above, to the CDF of a $\mathrm{Poisson}(2)$ random variable, pictured in black above.}
  \label{poisson2}
\end{figure}

\begin{figure}
  \includegraphics[width=\linewidth]{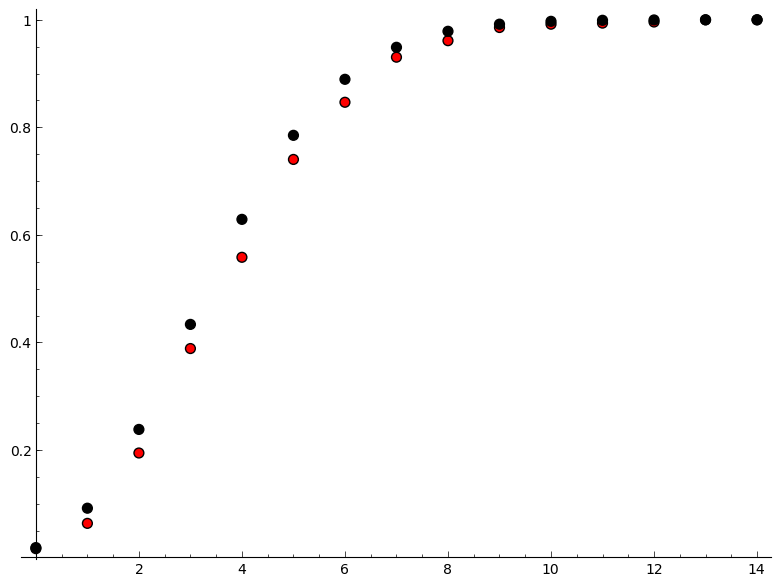}
  \caption{Comparing the CDF of the red dot distribution of $m_{\mathrm{new}}(N), 2500 \leq N \leq 5000,$ for $\omega(N) = 2$ from Figure \ref{simplelinearfactorslevelatmostthreeprimefactors}, pictured in red above, to the CDF of a $\mathrm{Poisson}(4)$ random variable, pictured in black above.}
  \label{poisson4}
\end{figure}

\begin{figure}
  \includegraphics[width=\linewidth]{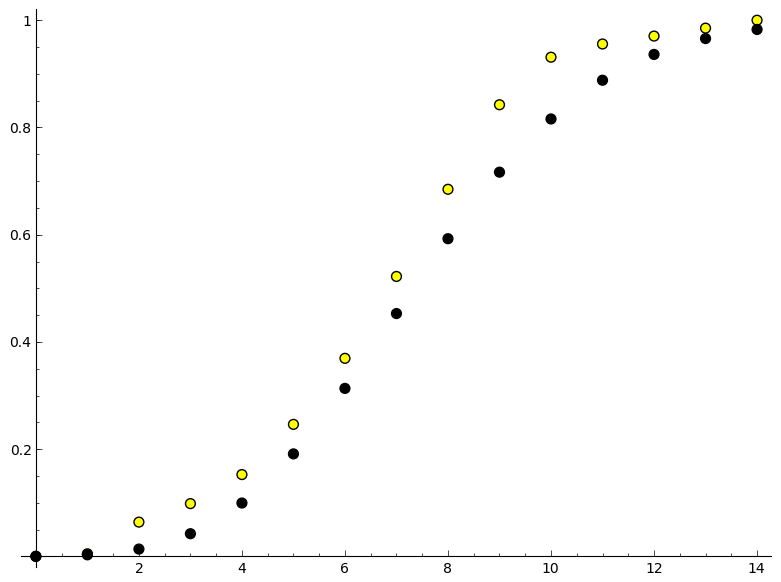}
  \caption{Comparing the CDF of the yellow dot distribution of $m_{\mathrm{new}}(N), 2500 \leq N \leq 5000,$ for $\omega(N) = 3$ from Figure \ref{simplelinearfactorslevelatmostthreeprimefactors}, pictured in yellow above, to the CDF of a $\mathrm{Poisson}(8)$ random variable, pictured in black above.}
  \label{poisson8}
\end{figure}

In Figures \ref{poisson2}, \ref{poisson4}, and \ref{poisson8}, we have overlaid the cumulative distribution function (CDF) of $m_{\mathrm{new}}(N)$ when $\omega(N) = s = 1,2,3$ in blue, red, and yellow respectively, and the CDF of $\mathrm{Poisson}(2^s)$ in black.  In all three cases, the discrepancy\footnote{The \emph{discrepancy} between two probability measures $\mu, \nu$ on $\R$ is defined to be 
$$\sup_x |\mathrm{CDF}(\mu)(x) - \mathrm{CDF}(\nu)(x)|$$} between the distribution of $m_{\mathrm{new}}(N)$ and $\mathrm{Poisson}(2^{\omega(N)})$ is quite small: \medskip

\begin{tabular}{l | c | c} 
$s$   & color & discrepancy between $m_{\mathrm{new}}(N), \omega(N) = s,$ and $\mathrm{Poisson}(2^s)$  \\
\hline
1 & blue & 0.05248 \ldots\\
\hline
2 & red & 0.07055 \ldots \\
\hline
3 & yellow & 0.12574 \ldots \\
\end{tabular}
\bigskip

\subsection{Reconciling the local-to-global Heuristics \ref{GL2bhargavaheuristictake1} and \ref{GL2bhargavaheuristictake2} with the random permutation Heuristic \ref{randompermutationmodel}}
 
Based on the limited data presented in this paper, it appears that the local-to-global Heuristics \ref{GL2bhargavaheuristictake1} and \ref{GL2bhargavaheuristictake2} might undercount $r(N),$ the number of conjugacy classes of homomorphisms $r: G_\Q \rightarrow \GL_2(\FF_p)$ of squarefree conductor $N$ satsfying local properties (a$'$),(b$'$),(c$'$),(d$'$), when $N$ has \emph{many prime factors}.  We speculate that if $U$ is a subset of the squarefree integers for which $\omega(N), N \in U,$ is bounded, then

\begin{equation} \label{boundednumberofGL2extensionsonaverage}
\frac{\sum_{N \in U \cap [1,X]} r(N)}{\# U \cap [1,X]} \text{ is bounded;}
\end{equation}

boundedness \eqref{boundednumberofGL2extensionsonaverage} is consistent with both the local-to-global Heuristic \ref{GL2bhargavaheuristictake1} and the random permutation Heuristic \ref{randompermutationmodel}.

Furthermore, we believe that the philosophy underlying the local-to-global heuristic, namely that the local components of $r:G_\Q \rightarrow \GL_2(\FF_p)$ should be equidistributed, is true and robust.  We close by describing one corroborating piece of data.

  \begin{figure} 
  \includegraphics[width=\linewidth]{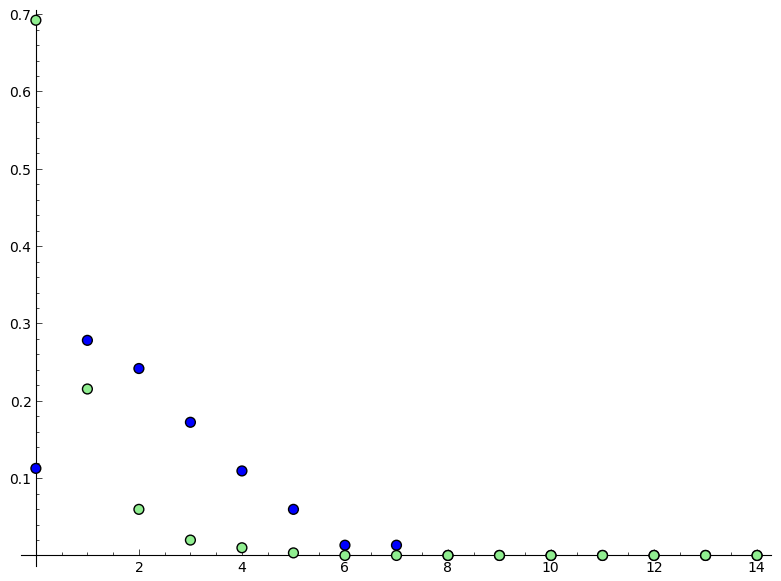}
  \caption{Comparing the distribution of $m_{\mathrm{new}}(N),$ pictured in blue, to the distribution of $m_{\mathrm{Hasse}}(N),$ pictured in light green, for prime $N$ in the interval $[2500,5000].$  Note the light green spike over $m_{\mathrm{Hasse}}(N) = 0.$}
  \label{hassebound}
\end{figure}

Let $m_{\mathrm{Hasse}}(N)$ denote the number of linear factors of the characteristic polynomial of $T_2$ acting on $S_2(N,\mathbb{F}_{101})$ \text{ satisfying the Hasse bound at 2}, i.e. whose roots are the reduction of an integer of absolute value at most $\lfloor 2 \sqrt{2} \rfloor = 2.$   As mentioned before, $m_{\mathrm{new}}(N)$ is a lower bound for $r(N),$ the number of conjugacy classes of representations $r: G_\Q \rightarrow \GL_2(\FF_{101})$ satisfying the local properties (a$'$),(b$'$),(c$'$),(d$'$).   Also, $m_{\mathrm{Hasse}}$ is an upper bound for the number of conjugacy classes of representations $r$ as above \emph{which also satisfy the Hasse bound at 2}; let $r_{\mathrm{Hasse}}(N)$ denote the number of such $r.$

Figure \ref{hassebound} plots the distribution of $m_{\mathrm{new}}(N),$ in blue, and the distribution of $m_{\mathrm{Hasse}}(N),$ in light green, for prime $N$ between 100 and 5000.  It is clear from the picture that the distribution of $m_{\mathrm{Hasse}}$ is far to the left of the distribution of $m_{\mathrm{new}}.$  Some summary statistics from this sample of 302 different primes $N$: \medskip

\begin{tabular}{l | c | c | c } 
color & random variable & mean  & variance  \\
\hline
blue & $m_{\mathrm{new}}$ & 2.18543 \ldots & 2.382834\ldots \\
\hline
light green & $m_{\mathrm{Hasse}}$ & 0.45033 \ldots & 0.67137 \ldots \\
\end{tabular}
\bigskip

Also note that for the sampled prime $N, m_{\mathrm{Hasse}}$ equals zero nearly $70\%$ of the time and is at most 1 more than $90\%$ of the time.  Because $r_{\mathrm{Hasse}}(N) \leq m_{\mathrm{Hasse}}(N)$ and $m_{\mathrm{new}}(N) \leq r(N),$ the distinction between $r_{\mathrm{Hasse}}(N)$ and $r(N)$ is even more pronounced than the distinction between the blue and light green dots in Figure \ref{hassebound}.  

\begin{appendix}
\section{Galois cohomology and the mass at $p$} \label{galoiscohomologyandmassp}
\subsection{Galois cohomology}
To facilitate our computation in \S \ref{massp} of the mass at $p$ for ordinary representations, we need to calculate $\# H^1(G_p, \chi \cdot e^2).$  The Cartier dual of $\chi \cdot e^2$ equals $\left( \chi^{-1} \cdot e^{-2} \right) \cdot \chi = e^{-2}.$  So by local duality for Galois cohomology,
$$H^2(G_p, \chi \cdot e^2) = H^0(G_p, e^{-2}).$$
\begin{itemize}
\item[(1)]
Suppose $e^2 = 1.$  

Then $\# H^0(G_p, \chi \cdot e^2) = 1, \# H^2(G_p,\chi \cdot e^2) = \# H^2(G_p, e^{-2}) = p.$ By the local Euler characteristic formula, 
$$\# H^1(G_p, \chi \cdot e^2) = p^2.$$

\item[(2)]
Suppose $e^2 \neq 1.$  

Then $\# H^0(G_p, \chi \cdot e^2) = 1$ because $\chi \cdot e^2 \neq 1$ (since $\chi$ has order $p-1$ and so is non-square) and $\# H^2(G_p, \chi \cdot e^2) = \# H^0(G_0, e^{-2}) = 1.$  By the local Euler characteristic formula,
$$\# H^1(G_p, \chi \cdot e^2) = p.$$
\end{itemize}

\subsection{The mass at $p$}
\label{massp}
\begin{itemize}
\item[(1$'$)]
Suppose $e^2 \neq 1.$  Then $\# \mathrm{Ext}^1(e^{-1}, \chi \cdot e) = p$ and all elements correspond to distinct $\GL_2(\FF_p)$-conjugacy classes of representations $r: G_p \rightarrow \GL_2(\FF_p).$ 

The centralizer of a non-trivial extension equals the center of $\GL_2(\FF_p).$  Indeed, there is some $g_0 \in G_p$ for which $e^{-1}(g_0) \neq \chi \cdot e(g_0)$ because $\chi$ is ramified and $e$ is unramified.  Since $r(g_0)$ has distinct eigenvalues in $\mathbb{F}_p^\times,$ we choose an eigenbasis $e_1,e_2.$  The centralizer of $r$ must commute with $r(g_0)$ and so must be diagonal with respect to $e_1,e_2.$  Equating the top right entries of the equality 
$$\left( \begin{array}{cc} \chi \cdot e & c \\ 0 & e^{-1} \end{array} \right) \left( \begin{array}{cc} x & 0 \\ 0 & y \end{array} \right) = \left( \begin{array}{cc} x & 0 \\ 0 & y \end{array} \right) \left( \begin{array}{cc} \chi \cdot e & c \\ 0 & e^{-1} \end{array} \right)$$
gives $yc = xc.$  Since $c$ is not the zero cocycle (it's not even a coboundary), we must have $x = y.$

There are $p-3$ characters $e$ for which $e^2 \neq 1.$  For every one of these characters, and every non-trivial element of $\mathrm{Ext}^1(e^{-1}, \chi \cdot e),$ there is a unique associated $\GL_2(\FF_p)$-conjugacy class of representations with mass $\frac{1}{p-1}.$  Such characters $e$ and $\mathrm{Ext}$ elements contribute 
$$\frac{1}{p-1}\cdot (p-3)(p-1)$$
to the mass at $p.$  For every trivial $\mathrm{Ext}$ class, i.e. for every split representation, the centralizer is a torus of order $\frac{1}{(p-1)^2}.$  Such trivial extensions contribute
$$\frac{1}{(p-1)^2} \cdot (p-3)$$
to the mass at $p.$

\item[(2$'$)]
Suppose $e^2 = 1.$  Then $\# \mathrm{Ext}^1(e^{-1}, \chi \cdot e) = p^2$ and all elements correspond to distinct $\GL_2(\FF_p)$ conjugacy classes of representations $r: G_p \rightarrow \GL_2(\FF_p).$  For the same reason as above, if $r$ is a non-trivial extension, i.e. if it corresponds to a non-trivial element of $\mathrm{Ext}^1(e^{-1}, \chi \cdot e),$ the centralizer of $r$ equals the center of $\GL_2(\FF_p).$  Therefore, the total mass at $p$ corresponding to non-trivial $\mathrm{Ext}$ classes equals
$$\frac{1}{(p-1)} \cdot 2 \cdot (p^2 - 1).$$
The centralizer of both split extensions is a torus of order $\frac{1}{(p-1)^2}.$  These two split extensions therefore contribute
$$\frac{1}{(p-1)^2} \cdot 2$$
to the mass at $p.$
\end{itemize}

Summing the contribution from (1$'$) and (2$'$), the total mass at $p$ equals
\begin{align*}
& \frac{1}{(p-1)} \cdot (p - 3)(p-1) +\frac{1}{(p-1)^2} \cdot (p-3) + \frac{1}{(p-1)} \cdot 2(p^2 - 1) + \frac{1}{(p-1)^2} \cdot 2 \\
= & 3p - 1 + \frac{1}{p-1}.
\end{align*}

\section{The order of magnitude of $\sum_{N \text{ squarefree } \in [X,2X]} 2^{\omega(N)}$} \label{tauberiantheorem}
Let 

\begin{equation*}
L(s) = \prod_{p \text{ prime}} (1 + 2 p^{-s}) = \sum_{N \text{ squarefree}} 2^{\omega(N)} N^{-s}
\end{equation*}

Note that 

\begin{align*}
\frac{L(s)}{\zeta(s)^2} &=\prod_p (1 + 2 p^{-s}) (1 - 2p^{-s} + p^{-2s})   \\
&= \prod_p (1 - 3p^{-2s}  + 2p^{-3s}),
\end{align*}

which converges absolutely and uniformly on $\Re(s) \geq \frac{1}{2} + \delta.$  Therefore, $L(s)$ is meromorphic on $\Re(s) > \frac{1}{2},$ has a pole of order 2 at $s=1,$ and is uniformly bounded on the vertical strips $\frac{1}{2} + \delta \leq \Re(s) \leq 1 - \delta.$  Let $f$ be a smooth function of compact support on $(0,\infty).$  Taking the Mellin transform and its inverse of the smoothed sum
$$S_f(X) = \sum_{N \text{ squarefree}} 2^{\omega(N)} f \left(  \frac{N}{X} \right)$$
gives
$$S_f(X) = \frac{1}{2\pi i} \int_{\Re(s) = c} L(s) f^\vee(s) \cdot X^s \; ds,$$

where $f^\vee$ denotes the Mellin transform of $f.$

Let $L(s) = \sum a_n (s-1)^n, f^\vee(s) = \sum b_n (s-1)^n$ be the Laurent expansions of $L$ and $f^\vee$ at $s = 1.$  Then 
$$L(s) f^\vee(s) \cdot X^s \text{ has residue } a_{-1} b_0 c_0 +a_{-2} b_1 \cdot  X + a_{-2} b_0 \cdot  X \log X \text{ at } s = 1.$$ 
Shifting the contour from $\Re(s) = 2$ to $\Re(s) = \frac{1}{2} + \delta,$ justified because $f^\vee(s)$ decays rapidly for $\Im(s)$ large, gives

$$S_f(X) = a_{-2} b_1 \cdot X + a_{-2} b_0 \cdot X \log X + O_{f,\delta}( X^{1/2 + \delta}),$$

where the constant in the big $O$ depends only on $f,\delta.$  Note also that

$$b_0 = f^\vee(1) = \int_0^{\infty} f(X) X^1 \frac{dX}{X} = \int_0^{\infty} f(X) dX.$$

Letting $f$ become an increasingly good approximation to the characteristic function of $[1,2],$ we find 

$$\sum_{N \text{ squarefree } \in [X,2X] } 2^{\omega(N)} \sim a_{-2} \cdot X \log X.$$

\end{appendix}

\end{document}